\documentclass{article}
\usepackage{amsfonts,amsmath,mathtools,amsthm,amssymb}
\numberwithin{equation}{section}

\theoremstyle{plain}
\newtheorem{defi}{Definition}[section]
\newtheorem{them}{Theorem}
\newtheorem{prop}{Proposition}[section]
\newtheorem{lemm}{Lemma}[section]
\newtheorem{corr}{Corollary}[section]

\author{Tang Qian}
\title{A Characterization of Umbral Calculus Inspired by Fractional Sums}

\begin{document}
\maketitle
\begin{abstract}
	We will use analytic function theory and Fourier analysis to establish a characterization for some classical umbral calculus, which will focus on the generalization of the evaluation function. Although we cannot cover all the umbral calculus people care about, the part about Bernoulli numbers can still answer an open question about fractional sums raised by Müller and Schleicher in 2005 \cite{Mueller2005} and synthesize some common analysis results. We will only develop the results which are sufficient to serve the purpose of this article, but at the same time, we will briefly mention some possible extensions.
\end{abstract}

\section{Introduction}
Umbral calculus is a basic technique in combinatorics, systematized by Gian-Carlo Rota \cite{Rota1994}. Generally speaking, for a given sequence $(A_n)_{n\in\mathbb{N}}$, it asks us to treat $A_n$ as $A^n$ to simplify some combinatorics deduction and revealed some imperceptible results. \cite{Gessel2001} gives an appropriate introduction about the applications, and \cite{DiBucchianico1995} gives a comprehensive survey. Consider Bernoulli numbers $(\mathrm{B}_n)_{n\in\mathbb{N}}$, we can write Faulhaber's formula as
\begin{equation*}
	\sum_{k=1}^{n} k^m=
	\frac{1}{m+1}\sum_{k=0}^{m}\binom{m+1}{k}\mathrm{B}_k n^{m+1-k}=
	\frac{(\mathrm{B}+n)^{m+1}-\mathrm{B}^{m+1}}{m+1}.
\end{equation*}
Therefore for polynomial $P$, we have
\begin{equation}
	\sum_{k=1}^{n} P(k)=
	\int_{0}^{n} P(\mathrm{B}+k) \mathrm{d}k.
	\label{umbralFaulhaber}
\end{equation}
Those results strongly hint that we can consider replacing the polynomial $P$ and the integer $n$ in \eqref{umbralFaulhaber} by function $f$ and complex number $z$. However, it requires us to explain $f(\mathrm{B})$ and $\sum_{k=1}^{z}$ appropriately , and we need more motivation for doing this.

Fractional sums is a concept first formulated clearly by Müller and Schleicher \cite{Mueller2005,Mueller2010,Mueller2011}. They proposed six natural axioms on $\sum_{k=a}^{b}$, for appropriate function $f$ and complex number $a,b$, those axioms can derive the value of $\sum_{k=a}^{b}f(k)$ uniquely. Their results can be reformulated as below.

\begin{defi}[Müller-Schleicher fractional summation]
	\label{MSdef}
	Suppose $U\subseteq \mathbb{C}$ satisfied $U+1\subseteq U$ and $1\in U$, $f$ is a complex value function on $U$.
	
	If there is a polynomial sequence $(P_n)_{n\in\mathbb{N}}$ satisfied
	\begin{enumerate}
		\item The degree of the polynomial sequence $(P_n)_{n\in\mathbb{N}}$ is uniformly bounded;
		\item For all $y+1\in U$, the limits
		\begin{equation*}
			\lim\limits_{n\to\infty}\left(
			\sum_{k=1}^{n}(f(k+y)-f(k))-P_n(y)
			\right)=-Q(y)
		\end{equation*}
		exists;
		\item For all $y+1\in U$,
		\begin{equation*}
			Q(y+1)-Q(y)=f(y+1),
		\end{equation*}
	\end{enumerate}
	then we defined \textbf{Müller-Schleicher fractional summation} as
	\begin{equation}
		(\mathrm{MS})\text{-}\sum_{k=x}^{y}f(k)
		\coloneqq
		Q(y)-Q(x-1),
	\end{equation}
	where $x,y+1\in U$.
\end{defi}
For the well-definedness of Müller-Schleicher fractional summation(for short, MS-fractional summation), we referred to \cite{Mueller2010}. Many classical infinite sum identities can be reformulated in 
fractional summation and revealed a direct proof \cite{Mueller2005}.

In those reformulations, especially for the Gosper series, which Gosper discussed and proved by the properties of Bessel function \cite{Gosper1993}.
\begin{align}
	\sum_{n=0}^{\infty}
	\frac{(-1)^n}{n+\frac{1}{2}}
	\frac{\sin\sqrt{b^2+\pi^2(n+1/2)^2}}{\sqrt{b^2+\pi^2(n+1/2)^2}}
	&=
	\frac{\pi \sin b}{2b}
	\label{gospersin}\\
	\sum_{n=1}^{\infty}
	\frac{(-1)^n}{n^2}
	\cos\sqrt{b^2+\pi^2 n^2}
	&=
	\frac{\pi^2}{4}\left(
	\frac{\sin b}{b}-\frac{\cos b}{3}\right).
	\label{gospercos}
\end{align}
They noticed that \cite{Mueller2005,Mueller2010,Mueller2011} those series could be proved directly by an interchange of a fractional sum and an infinite series, i.e.
\begin{align*}
	\sum_{n=0}^{\infty}
	\frac{(-1)^n}{n+\frac{1}{2}}
	\frac{\sin\sqrt{b^2+\pi^2(n+1/2)^2}}{\sqrt{b^2+\pi^2(n+1/2)^2}}
	&=
	(\mathrm{MS})\text{-}
	\sum_{n=1/4}^{-1/4}
	\frac{1}{2n}
	\frac{\sin\sqrt{b^2+(2\pi n)^2}}{\sqrt{b^2+(2\pi n)^2}}\\\
	&\overset{?}{=}
	(\mathrm{MS})\text{-}
	\frac{\sin b}{2b}
	\sum_{n=1/4}^{-1/4}n^{-1},\\
\end{align*}
and
\begin{align*}
	\sum_{n=1}^{\infty}
	\frac{(-1)^n}{n^2}
	\cos\sqrt{b^2+\pi^2 n^2}
	&=
	(\mathrm{MS})\text{-}
	\frac{1}{4}
	\sum_{n=1}^{-1/2}
	\frac{1}{n^2}
	\cos\sqrt{b^2+(2\pi n)^2}\\
	&\overset{?}{=}
	(\mathrm{MS})\text{-}
	\frac{1}{4}\left(
	\cos b \sum_{n=1}^{-1/2} n^{-2}-
	\frac{2\pi^2 \sin b}{b}
	\sum_{n=1}^{-1/2}1
	\right).
\end{align*}
Where
\begin{align*}
	\frac{1}{2n}
	\frac{\sin\sqrt{b^2+(2\pi n)^2}}{\sqrt{b^2+(2\pi n)^2}}
	&=
	\frac{\sin b}{2b}n^{-1}+
	a_1 n+a_3 n^3+a_5 n^5+\cdots\\
	\frac{1}{n^2}
	\cos\sqrt{b^2+(2\pi n)^2}
	&=
	\cos b\cdot n^{-2}-
	\frac{2\pi^2 \sin b}{b}+
	a_2 n^2+a_4 n^4+\cdots,
\end{align*}
and we have \[
(\mathrm{MS})\text{-}
\sum_{n=1/4}^{-1/4}n^{2k-1}=0\quad
(\mathrm{MS})\text{-}\sum_{n=1}^{-1/2}n^{2k}=0
\]
\[
(\mathrm{MS})\text{-}
\sum_{n=1/4}^{-1/4}n^{-1}=\pi\quad
(\mathrm{MS})\text{-}
\sum_{n=1}^{-1/2}n^{-2}=-\frac{\pi^2}{3}\quad
(\mathrm{MS})\text{-}
\sum_{n=1}^{-1/2}1=-\frac{1}{2},
\] for all positive integer $k$. We have enough motivation to determine the sufficient condition on function for this interchangeability. In \cite{Mueller2010} they also concerned about the conditions that validated the following formula
\begin{equation}
	(\mathrm{MS})\text{-}
	\frac{\mathrm{d}}{\mathrm{d}x}
	\sum_{k=1}^{x}f(k)
	=c_f+
	(\mathrm{MS})\text{-}
	\sum_{k=1}^{x}
	\frac{\mathrm{d}}{\mathrm{d}k}f(k).
\end{equation}
Our goal is to justify those speculation.

\subsection{Framework}

Our main idea is generalized \eqref{umbralFaulhaber} to
\begin{equation}
	(\mathrm{MS})\text{-}
	\sum_{k=1}^{z} f(k)=
	F(\mathrm{B}+z)-F(\mathrm{B}),
\end{equation}
where $F$ is the primitive function of $f$, and transferred the difficulty of interchangeability to the definition and properties of $F(\mathrm{B})$, i.e. $\mathrm{eval}_{\mathrm{B}}(F)$. We achieved this in Theorem \ref{MScomp}.

In the next section, we will discuss the evaluation function including a kind of sequence $(A_n)_{n\in\mathbb{N}}$, such as Bernoulli numbers, and developed all the properties we need. This is achieved by the analytic function theory and Fourier transform, i.e.
\begin{equation}
	f(A)\coloneqq\mathrm{eval}_A(f)\coloneqq
	\frac{1}{\sqrt{2\pi}}
	\int_{-\infty-\mathrm{i}t}^{\infty-\mathrm{i}t}
	\hat{\mathcal{A}}(z)f(\mathrm{i}z)\mathrm{d}z,
\end{equation}
where $\mathcal{A}(z)$ is the generating function of sequence $(A_n)_{n\in\mathbb{N}}$. In fact, similar ideas can be founded in other places. In \cite{Ismail1997,Ismail1998}, the author treats $A_n$ as the $n$-th moment of some signed measure $\mu$; in \cite{Zeilberger1980}, the author also noticed connections with Fourier analysis; in \cite{Grabiner1988,Grabiner1989}, the author gives an extension of the umbral calculus to certain classes of entire function. Our discussion will follow a different way.

Section 3 will focus on the applications of section 2, Corollary \ref{MSinter}, \ref{MSinterderi} includes the validity of the speculation mentioned above and determined $c_f$. Finally, in Corollary \ref{Gosperseries}, we proved and promoted some conclusions in \cite{Gosper1993} as expected by Müller and Schleicher.

In Section 4 we introduced some possible developments.

\section{The Analysis of Umbrae}

The Fourier transform of $f\in L^1(\mathbb{R})$, which we take is
\begin{equation*}
	\hat{f}(\xi)\coloneqq
	\frac{1}{\sqrt{2\pi}}
	\int_{-\infty}^{\infty}
	f(x)\mathrm{e}^{-\mathrm{i}\xi x}\mathrm{d}x,
\end{equation*}
and it also induced the Fourier transform on slowly increasing distribution space $\mathcal{S}'(\mathbb{R})$. Sometimes we would interpret the integral of the Fourier transform in the meaning of the summation method.

For our purpose, we will generalize the concept of the Fourier transform as below.
\begin{defi}[Fourier transform]
	\label{genft}
	If $f$ is defined on $\mathbb{R}-\mathrm{i}t$, then we define the Fourier transform of $f$ is
	\begin{align}
		\hat{f}(\xi-\mathrm{i}s)
		&\coloneqq
		\frac{1}{\sqrt{2\pi}}
		\int_{-\infty-\mathrm{i}t}^{\infty-\mathrm{i}t}
		f(z)\mathrm{e}^{-\mathrm{i}(\xi-\mathrm{i}s) z}\mathrm{d}z\\
		&=
		\frac{\mathrm{e}^{-(\xi-\mathrm{i}s)t}}{\sqrt{2\pi}}
		\int_{-\infty}^{\infty}
		\mathrm{e}^{-sx}f(x-\mathrm{i}t)\mathrm{e}^{-\mathrm{i}\xi x}\mathrm{d}x.
		\notag
	\end{align}
\end{defi}
\noindent
It should be emphasized that $f$ could be exponential growth, and $\hat{f}$ could be undefined on $\mathbb{R}$.
\begin{defi}[Umbrae]
	If $\mathcal{A}$ is an \textbf{exponential type} analytic function defined on
	\begin{equation*}
		\Omega_{a,b}\coloneqq
		\{
		x-\mathrm{i}t:x\in\mathbb{R},t\in(a,b),
		\}
	\end{equation*}
	i.e. there exists $s\in\mathbb{R}$ such that for every $t\in(a,b)$,
	\begin{equation*}
		\left|\mathcal{A}(x-\mathrm{i}t)\right|
		\leq C(t)\mathrm{e}^{s|x|},
	\end{equation*}
	where $C$ is locally bounded, then
	\begin{enumerate}
		\item we called $A=(\mathcal{A},\Omega_{a,b})$ is a \textbf{umbrae};
		\item we called $\mathcal{A}$ is the \textbf{generating function} of $A$.
	\end{enumerate}
\end{defi}
\noindent
When $0\in\Omega_{a,b}$, the umbrae $A$ exactly corresponds to the sequence $(\mathcal{A}^{(n)}(0))_{n\in\mathbb{N}}$. $s$ can be treated as dependent on $t$, but we do not need such a detailed treatment here.

\begin{lemm}
	\label{welldeflem}
	Suppose $f$ is an exponential type analytic function defined on $\Omega_{a,b}$. $t_1,t_2\in(a,b)$.
	\newline
	If $f$ is Lebesgue integrable on $\mathbb{R}-\mathrm{i}t_1,\mathbb{R}-\mathrm{i}t_2$, then
	\begin{equation}
		\int_{-\infty-\mathrm{i}t_1}^{\infty-\mathrm{i}t_1}
		f(z)\mathrm{d}z
		=\int_{-\infty-\mathrm{i}t_2}^{\infty-\mathrm{i}t_2}
		f(z)\mathrm{d}z.
	\end{equation}
\end{lemm}
\begin{proof}
	First notice that $\mathrm{e}^{-\varepsilon z^2}f(z)$ locally uniformly tends to $0$ when $|z|\to\infty$, therefore
	\begin{equation*}
		\int_{-\infty-\mathrm{i}t_1}^{\infty-\mathrm{i}t_1}
		\mathrm{e}^{-\varepsilon z^2}f(z)\mathrm{d}z
		=\int_{-\infty-\mathrm{i}t_2}^{\infty-\mathrm{i}t_2}
		\mathrm{e}^{-\varepsilon z^2}f(z)\mathrm{d}z.
	\end{equation*}
	By the dominated convergence theorem, only need to let $\varepsilon\to0$.
\end{proof}
\noindent
This lemma is also valid for larger function classes, but again the exponential type condition is sufficient for our needs.

\begin{lemm}[$L^1$-Phragmén–Lindelöf principle]
	\label{L1-PL}
	Suppose $f$ is an exponential type analytic function defined on $\Omega_{a,b}$. $t_1,t_2\in(a,b),t_1<t_2$.
	\newline
	If $f$ is Lebesgue integrable on $\mathbb{R}-\mathrm{i}t_1,\mathbb{R}-\mathrm{i}t_2$, then for every $t\in[t_1,t_2]$, $f$ is Lebesgue integrable on $\mathbb{R}-\mathrm{i}t$, and $\ln \Vert f(\cdot-\mathrm{i}t)\Vert_1$ is a convex function for $t\in[t_1,t_2]$.
\end{lemm}
\begin{proof}
	See \cite{Ehrenpreis1958} p.479 Theorem 4.
\end{proof}

\subsection{Basic definition and properties}

We basically only care about a special class of umbrae.
\begin{defi}
	Suppose $A=(\mathcal{A},\Omega_{a,b})$ is a umbrae.
	\begin{enumerate}
		\item Denote
		\begin{align*}
			\alpha &= \inf \{
			s\in\mathbb{R}:
			\exists \text{locally bounded C}.
			\forall x\geq 0.|f(x-\mathrm{i}t)|\leq C(t)\mathrm{e}^{sx}
			\}\\
			\beta &= \sup \{
			s\in\mathbb{R}:
			\exists \text{locally bounded C}.
			\forall x\leq 0.|f(x-\mathrm{i}t)|\leq C(t)\mathrm{e}^{sx}
			\}.
		\end{align*}
		We called $\alpha$ is the \textbf{positive index}; $\beta$ is the \textbf{negative index}; $(\alpha,\beta)$ is the \textbf{index} of the umbrae $A$;
		\item If $\alpha\geq\beta$, we called umbrae $A$ is \textbf{singular};
		\item If $\alpha<\beta$, we called umbrae $A$ is \textbf{regular}, and the open interval $(\alpha,\beta)$ is called the \textbf{regular interval} of $A$;
		\item We called the open interval $(a,b)$ is the \textbf{dominating interval} of $A$.
	\end{enumerate}
\end{defi}
\noindent
The reasons for choosing these terms and signs will become clearer in the following discussion.
\begin{defi}[Umbrae calculus]
	Suppose $A_i=(\mathcal{A}_i,\Omega_{a_i,b_i})$ is a umbrae.
	\begin{enumerate}
		\item
		For $r\in\mathbb{R}$, $rA\coloneqq(\mathcal{A}(rz),r^{-1}\Omega_{a,b})$;
		\item $A_1+A_2\coloneqq(\mathcal{A}_1\mathcal{A}_2,\Omega_{a_1,b_1}\cap\Omega_{a_2,b_2})$, 
		$A_1-A_2\coloneqq A_1+(-1)A_2$;
		\item For $n\in\mathbb{N}$,
		$0\times A\coloneqq 0,(n+1)\times A\coloneqq (n\times A)+A$;
		\item
		$A_1[+]A_2\coloneqq(\mathcal{A}_1+\mathcal{A}_2,\Omega_{a_1,b_1}\cap\Omega_{a_2,b_2})$;
		\item
		$A_1[-]A_2\coloneqq(\mathcal{A}_1-\mathcal{A}_2,\Omega_{a_1,b_1}\cap\Omega_{a_2,b_2})$.
	\end{enumerate}
\end{defi}
\noindent
The scalar multiplication and addition in the above definition come from the classic umbral calculus. They do not satisfy the properties that these operations should normally have, but this definition is convenient for evaluation, and we will verify its well-definedness later.

\begin{defi}[Special umbrae]~ 
	\begin{enumerate}
		\item For $c\in\mathbb{C}$, $(c)\coloneqq(\mathrm{e}^{cz},\mathbb{C})$,
		the index is $(\mathrm{Re} c,\mathrm{Re} c)$;
		\item For $c\in\mathbb{C}$, $[c]\coloneqq(c,\mathbb{C})$,
		the index is
		$\begin{cases*}
			(0,0), &$c\neq 0$ \\
			(-\infty,\infty), &$c=0$
		\end{cases*}$
		\item $\mathrm{D}\coloneqq(z,\mathbb{C})$,
		the index is $(0,0)$;
		\item $\Delta\coloneqq(\mathrm{e}^z-1,\mathbb{C})$,
		the index is $(1,0)$;
		\item $\mathrm{B}\coloneqq(\frac{z\mathrm{e}^z}{\mathrm{e}^z-1},\Omega_{-2\pi,2\pi})$, the index is $(0,1)$;
		\item
		$\mathrm{E}\coloneqq(\frac{2}{\mathrm{e}^z+\mathrm{e}^{-z}},\Omega_{-\frac{\pi}{2},\frac{\pi}{2}})$, the index is $(-1,1)$.
	\end{enumerate}
\end{defi}
\noindent
In fact, we will identify the complex number $c$ to the umbrae $(c)$.

For regular umbrae $A=(\mathcal{A},\Omega_{a,b})$, we can use Definition \ref{genft} to calculate the Fourier transform of $\mathcal{A}$. By the knowledge of analysis, it is not difficult to see that $\hat{\mathcal{A}}$ is well-defined, and it also induced a regular umbrae.
\begin{them}[Correspondence]
	\label{corft}
	If $A=(\mathcal{A},\Omega_{a,b})$ is a regular umbrae with regular interval $(\alpha,\beta)$, then $\hat{A}\coloneqq(\hat{\mathcal{A}},\Omega_{\alpha,\beta})$ is a regular umbrae with regular interval includes $(-b,-a)$.
\end{them}
\noindent
Fourier transform establishes the correspondence between regular umbrae.

Since we essentially only can deal with the regular umbrae, it is necessary to establish a decomposition theorem for the singular one.
\begin{them}[Component]
	\label{component}
	Suppose $A=(\mathcal{A},\Omega_{a,b})$ is a singular umbrae with index $(\alpha,\beta)$.
	If there are regular umbrae pairs $(A_1,A_2)$ inducing the decomposition $A=A_1[+]A_2$, then one of the regular umbrae $A_i$ has positive index $\alpha$, called the \textbf{positive component} of this decomposition, and the other $A_j$ has negative index $\beta$, its opposite $[-1]+A_j$ called the \textbf{negative component} of this decomposition. $(A_i,[-1]+A_j)$ is called the \textbf{regular decomposition} of $A$.
\end{them}
\begin{them}[Decomposition]
	\label{decompum}
	If $A=(\mathcal{A},\Omega_{a,b})$ is a umbrae with index $(\alpha,\beta)$, then there are umbraes $A^+,A^-$ such that
	\begin{enumerate}
		\item $A^+=(\mathcal{A}^+,\Omega_{a,b})$ is a regular umbrae with index $(\alpha,\infty)$;
		\item $A^-=(\mathcal{A}^-,\Omega_{a,b})$ is a regular umbrae with index $(-\infty,\beta)$;
		\item $A=A^+[-]A^-$.
	\end{enumerate}
\end{them}
\begin{proof}
	We denoted
	\begin{align*}
		\Lambda_+&=(\frac{1}{\mathrm{i}\sqrt{2\pi}}z^{-1}\mathrm{e}^{-z^2},\Omega_{0,\infty})
		=(\rho_+,\Omega_{0,\infty});\\
		\Lambda_-&=(\frac{1}{\mathrm{i}\sqrt{2\pi}}z^{-1}\mathrm{e}^{-z^2},\Omega_{-\infty,0})
		=(\rho_-,\Omega_{-\infty,0}).
	\end{align*}
	According to Theorem \ref{corft},
	\begin{enumerate}
		\item $\check{\Lambda}_+=(\check{\rho}_+,\Omega_{-\infty,\infty})$ 
		is a  regular umbrae with index $(0,\infty)$;
		\item $\check{\Lambda}_-=(\check{\rho}_-,\Omega_{-\infty,\infty})$ 
		is a  regular umbrae with index $(-\infty,0)$;
		\item $\check{\Lambda}_+[-]\check{\Lambda}_-=0$.
	\end{enumerate}
	Finally, we can take $A^+=A+\check{\Lambda}_+,A^-=A+\check{\Lambda}_-$.
\end{proof}
\noindent
This theorem allows us to transfer the problem of singular umbrae into regular umbraes.

Next, we define a class of functions that the umbrae can be evaluated. This function class will not be broad enough, but it is enough for our purposes.
\begin{defi}
	Suppose $A=(\mathcal{A},\Omega_{a,b})$ is a regular umbrae, the index is $(\alpha,\beta)$. $I$ is an open interval.
	If $I\cap(\alpha,\beta)\neq\varnothing$, and the analytic function $f$ satisfies 
	\begin{enumerate}
		\item 
		$f$ is an exponential type analytic function defined on $\{z\in\mathbb{C}:\mathrm{Re}z\in I\}$;
		\item For every $t\in I\cap(\alpha,\beta)$, $\hat{\mathcal{A}}(z)f(\mathrm{i}z)$ is Lebesgue integrable on $\mathbb{R}-\mathrm{i}t$,
	\end{enumerate}
	then the class formed by such functions is denoted by $\mathcal{T}_A(I)$.
\end{defi}
\begin{defi}[Evaluation function]
	\label{Evaldef}
	If $f\in\mathcal{T}_A(I)$, then we defined
	\begin{equation}
		f(A)\coloneqq\mathrm{eval}_A(f(z);z)
		\coloneqq \int_{-\infty-\mathrm{i}t}^{\infty-\mathrm{i}t}
		\hat{\mathcal{A}}(z)f(\mathrm{i}z)\mathrm{d}z,
	\end{equation}
	where $t\in I\cap(\alpha,\beta)$.
\end{defi}
\noindent
Lemma \ref{welldeflem} ensured the well-definedness of $f(A)$.

In fact, the argument from Lemma \ref{welldeflem} can also tell us the following result about summation method.
\begin{prop}[Gauss-Weierstrass summation method]
	Suppose $A=(\mathcal{A},\Omega_{a,b})$ is a regular umbrae, the index is $(\alpha,\beta)$. $f$ is an exponential type analytic function defined on $\{z\in\mathbb{C}:\mathrm{Re}z\in \tilde{I}\}$. $I,\tilde{I}$ are open intervals satisfied $I\subseteq \tilde{I}$.
	
	If $f\in\mathcal{T}_A(I)$, then for every $t\in\tilde{I}\cap(\alpha,\beta)$ we have
	\begin{equation}
		\label{GWsummation}
		\lim\limits_{\varepsilon\to 0}
		\int_{-\infty-\mathrm{i}t}^{\infty-\mathrm{i}t}
		\mathrm{e}^{-\varepsilon z^2}
		\hat{\mathcal{A}}(z)f(\mathrm{i}z)\mathrm{d}z
		=f(A).
	\end{equation}
\end{prop}
\begin{proof}
	Because of
	\begin{equation*}
		\int_{-\infty-\mathrm{i}t}^{\infty-\mathrm{i}t}
		\mathrm{e}^{-\varepsilon z^2}
		\hat{\mathcal{A}}(z)f(\mathrm{i}z)\mathrm{d}z
		=
		\int_{-\infty-\mathrm{i}t_0}^{\infty-\mathrm{i}t_0}
		\mathrm{e}^{-\varepsilon z^2}
		\hat{\mathcal{A}}(z)f(\mathrm{i}z)\mathrm{d}z,
	\end{equation*}
	only need to let $\varepsilon\to0$.
\end{proof}
\noindent
We can also use \eqref{GWsummation} to define the evaluation function, but it will complicate the limit calculus.

\subsection{Well-definedness}

Intuitively we have $f(A_1[-]A_2)=f(A_1)-f(A_2)$, which inspired the following proposition.

\begin{prop}
	\label{twotemprop}
	Suppose $A=(\mathcal{A},\Omega_{a,b})$ is a singular umbrae with index $(\alpha,\beta)$. $I$ is an open interval.
	If $I\supseteq[\beta,\alpha]$, and the following conditions are satisfied,
	\begin{enumerate}
		\item $(A_1^+,A_1^-),(A_2^+,A_2^-)$ is the regular decomposition of $A$;
		\item $f\in
		\left(\mathcal{T}_{A_1^+}(I)\cap\mathcal{T}_{A_1^-}(I)\right)
		\cap
		\left(\mathcal{T}_{A_2^+}(I)\cap\mathcal{T}_{A_2^-}(I)\right)$,
	\end{enumerate}
	 then
	 \begin{enumerate}
	 	\item $f(A_1^+)-f(A_1^-)=f(A_2^+)-f(A_2^-)$;
	 	\item $A_0=A_1^+[-]A_2^+=A_1^-[-]A_2^-$ is a regular umbrae;
	 	\item $f\in\mathcal{T}_{A_0}(I)$.
	 \end{enumerate}
\end{prop}
\begin{proof}
	According to Theorem \ref{component}, the regular interval of the component of the regular decomposition always intersects with $I$, and $A_0=A_1^+[-]A_2^+=A_1^-[-]A_2^-$ is a regular umbrae.
	
	Notice that
	\begin{align*}
		f\in\mathcal{T}_{A_1^+}(I)\cap\mathcal{T}_{A_2^+}(I)
		&\Rightarrow f\in\mathcal{T}_{A_0}(I\cap(\alpha,\alpha+\varepsilon))\\
		f\in\mathcal{T}_{A_1^-}(I)\cap\mathcal{T}_{A_2^-}(I)
		&\Rightarrow f\in\mathcal{T}_{A_0}(I\cap(\beta-\varepsilon,\beta)).
	\end{align*}
	Since $f$ is exponential type, $f(A_1^+)-f(A_2^+)=f(A_1^-)-f(A_2^-)$ holds from Lemma \ref{welldeflem}. Lemma \ref{L1-PL} ensured that $f\in\mathcal{T}_{A_0}(I)$.
\end{proof}
\begin{defi}
	\label{singdef}
	Suppose $A=(\mathcal{A},\Omega_{a,b})$ is a singular umbrae with index $(\alpha,\beta)$. $I$ is an open interval includes $[\beta,\alpha]$.

	If there is a regular decomposition $(A^+,A^-)$ of $A$ such that $f\in\mathcal{T}_{A^+}(I)\cap\mathcal{T}_{A^-}(I)$, then we defined $f(A)\coloneqq f(A^+)-f(A^-)$.
\end{defi}
\noindent
Here we are not going to give a complete definition of $f(A)$ for the singular umbrae $A$. The general form of Proposition \ref{twotemprop} can be derived by Lemma \ref{welldeflem} and Theorem \ref{decompum}.

\begin{prop}
	\label{welldefeval}
	Suppose $A_1=(\mathcal{A}_1,\Omega_{a_1,b_1}),A_2=(\mathcal{A}_2,\Omega_{a_2,b_2})$ are regular umbrae with index $(\alpha_1,\beta_1),(\alpha_2,\beta_2)$.
	
	If $f(z_1,z_2)$ satisfied
	\begin{enumerate}
		\item For every $\mathrm{Re} z_2\in I_2$ we have $f(\cdot,z_2)\in\mathcal{T}_{A_1}(I_1)$;
		\item $f(A_1,\cdot)\in\mathcal{T}_{A_2}(I_2)$;
		\item For every $\mathrm{Re} z_1\in I_1$ we have $f(z_1,\cdot)\in\mathcal{T}_{A_2}(I_2)$;
		\item $f(\cdot,A_2)\in\mathcal{T}_{A_1}(I_1)$;
		\item There exists $t_1\in I_1\cap(\alpha_1,\beta_1),t_2\in I_2\cap(\alpha_2,\beta_2)$ such that $\hat{\mathcal{A}}_1(z_1)\hat{\mathcal{A}}_2(z_2)f(\mathrm{i}z_1,\mathrm{i}z_2)$ is Lebesgue integrable on $(\mathbb{R}-\mathrm{i}t_1)\times(\mathbb{R}-\mathrm{i}t_2)$,
	\end{enumerate}
	then we have
	\begin{equation}
		\mathrm{eval}_{A_2}(\mathrm{eval}_{A_1}(f(z_1,z_2);z_1);z_2)=
		\mathrm{eval}_{A_1}(\mathrm{eval}_{A_2}(f(z_1,z_2);z_2);z_1),
	\end{equation}
	i.e. $f(A_1,A_2)$ is well-defined.
\end{prop}
\begin{proof}
	Apply the Fubini theorem.
\end{proof}
\begin{prop}
	\label{welldefplus}
	Suppose $A_1=(\mathcal{A}_1,\Omega_{a,b}),A_2=(\mathcal{A}_2,\Omega_{a,b})$ are regular umbrae with index $(\alpha_1,\beta_1),(\alpha_2,\beta_2)$.
	
	If $g(z_1,z_2)=f(z_1+z_2)$ satisfies the conditions of Proposition \ref{welldefeval}, and there is an open interval $I\subseteq I_1+I_2$ such that $f\in\mathcal{T}_{A_1+A_2}(I)$, then
	\begin{equation*}
		f(A_1+A_2)=g(A_1,A_2),
	\end{equation*}
	i.e. $f(A_1+A_2)$ is well-defined.
\end{prop}
\begin{proof}
	We assume that $s\in(a,b)$. First notice that
	\begin{align*}
		g(A_1,A_2)&=\frac{1}{2\pi}
		\int_{-\infty-\mathrm{i}t_1}^{\infty-\mathrm{i}t_1}
		\int_{-\infty-\mathrm{i}t_2}^{\infty-\mathrm{i}t_2}
		\hat{\mathcal{A}}_1(z_1)\hat{\mathcal{A}}_2(z_2)
		f(\mathrm{i}z_1+\mathrm{i}z_2)\mathrm{d}z_2\mathrm{d}z_1\\
		&=\frac{1}{2\pi}
		\int_{-\infty-\mathrm{i}(t_1+t_2)}^{\infty-\mathrm{i}(t_1+t_2)}
		\int_{-\infty-\mathrm{i}t_2}^{\infty-\mathrm{i}t_2}
		\mathrm{e}^{s(z_0-z_2)}\hat{\mathcal{A}}_1(z_0-z_2)
		\mathrm{e}^{sz_2}\hat{\mathcal{A}}_2(z_2)
		\mathrm{e}^{-sz_0}f(\mathrm{i}z_0)\mathrm{d}z_2\mathrm{d}z_0\\
		&=\frac{1}{\sqrt{2\pi}}
		\int_{-\infty-\mathrm{i}(t_1+t_2)}^{\infty-\mathrm{i}(t_1+t_2)}
		\mathrm{e}^{-sz_0}
		(\mathrm{e}^{s(\cdot)}\hat{\mathcal{A}}_1
		\ast\mathrm{e}^{s(\cdot)}\hat{\mathcal{A}}_2)(z_0)
		f(\mathrm{i}z_0)\mathrm{d}z_0,
	\end{align*}
	and then notice that $f$ is an exponential type analytic function defined on $\{z\in\mathbb{C}:\mathrm{Re} z\in I_1+I_2\}$, so the conclusion followed by Lemma \ref{welldeflem}.
\end{proof}
\begin{prop}
	Suppose $A=(\mathcal{A},\Omega_{a,b})$ is a regular umbrae with index $(\alpha,\beta)$, $r\in\mathbb{R}$. 
	
	If $g(z)=f(rz)$, then we have
	\begin{equation*}
		f\in\mathcal{T}_{rA}(I)\Leftrightarrow
		g\in\mathcal{T}_{A}(r^{-1}I),
	\end{equation*}
	meanwhile $g(A)=f(rA)$ i.e. $f(rA)$ is well-defined.
\end{prop}
Finally, we will be concerned about the well-definedness related to some special singular umbrae, such as $(c)$ and $\mathrm{D}$. We hope they act like a constant and derivative.
\begin{prop}
	Suppose $A=(\mathcal{A},\Omega_{a,b})$ is a regular umbrae with index $(\alpha,\beta)$, $c\in\mathbb{C}$. 
	
	If $g(z)=f(z+c)$, then we have
	\begin{equation*}
		f\in\mathcal{T}_{A+(c)}(I)\Leftrightarrow
		g\in\mathcal{T}_{A}(I-c),
	\end{equation*}
	meanwhile $g(A)=f(A+(c))$ i.e. $f(A+c)$ is well-defined.
\end{prop}
Unfortunately, the derivative calculus is not completely well-defined. However, understanding this counterexample can also increase our understanding of umbrae calculus. Consider $f(z)=1, A=(z^{-1}\mathrm{e}^{-z^2},\Omega_{0,\infty})$, we have $f\in\mathcal{T}_{A+\mathrm{D}}(\mathbb{R}), f'\in\mathcal{T}_{A}(\mathbb{R})$, but $f(A+\mathrm{D})=1, f'(A)=0$.
\begin{them}
	\label{intercderi}
	Suppose $A=(\mathcal{A},\Omega_{a,b})$ is a regular umbrae with index $(\alpha,\beta)$. 
	
	If $\int_{-\infty-\mathrm{i}t}^{\infty-\mathrm{i}t}
	|\hat{\mathcal{A}}(z)f(\mathrm{i}z+z_0)|\mathrm{d}z$ is uniformly bounded for sufficiently small $z_0$, and $f(\cdot+z_0)\in\mathcal{T}_A(I), t\in I$, then we have
	\begin{enumerate}
		\item $f'\in \mathcal{T}_A(I)\Rightarrow 
		f'(A)=\frac{\mathrm{d}}{\mathrm{d}z_0}f(A+z_0)(0)$;
		\item $f\in \mathcal{T}_{A+\mathrm{D}}(I)\Rightarrow 
		f(A+\mathrm{D})=\frac{\mathrm{d}}{\mathrm{d}z_0}f(A+z_0)(0)$,
	\end{enumerate}
	i.e. $f(A+\mathrm{D})$ is well-defined.
\end{them}
\begin{proof}
	Suppose that $f'\in \mathcal{T}_A(I)$. For
	\begin{equation*}
		g_{\varepsilon}(z_0)=
		\frac{1}{\sqrt{2\pi}}
		\int_{-\infty-\mathrm{i}t}^{\infty-\mathrm{i}t}
		\mathrm{e}^{-\varepsilon z^2}
		\hat{\mathcal{A}}(z)f(\mathrm{i}z+z_0)\mathrm{d}z,
	\end{equation*}
	we have
	\begin{enumerate}
		\item $\lim\limits_{\varepsilon\to 0}g_{\varepsilon}(z_0)=f(A+z_0)$;
		\item Sequence $(g_\varepsilon)_{\varepsilon>0}$ is uniformly bounded.
	\end{enumerate}
	By the Montel's theorem, there exists a subsequence $(g_{\varepsilon_n})_{n\in\mathbb{N}}$ locally uniformly converges to $f(A+z_0)$. $f'$ is again an exponential type analytic function, by the Leibniz integral rule we have
	\begin{equation*}
		g'_{\varepsilon}(z_0)=
		\frac{1}{\sqrt{2\pi}}
		\int_{-\infty-\mathrm{i}t}^{\infty-\mathrm{i}t}
		\mathrm{e}^{-\varepsilon z^2}
		\hat{\mathcal{A}}(z)f'(\mathrm{i}z+z_0)\mathrm{d}z.
	\end{equation*}
	Because $g'_{\varepsilon}(0)\to f'(A), g'_{\varepsilon_n}(0)\to \frac{\mathrm{d}}{\mathrm{d}z_0}f(A+z_0)(0).$
	
	Suppose that $f\in \mathcal{T}_{A+\mathrm{D}}(I)$, and $t+z_0\in I$. Only need to notice that
	\begin{align*}
		\sqrt{2\pi}
		\tilde{g}_{\varepsilon}(z_0)&=
		\int_{-\infty-\mathrm{i}t}^{\infty-\mathrm{i}t}
		\mathrm{e}^{-\varepsilon (z-\mathrm{i}z_0)^2}
		\hat{\mathcal{A}}(z)f(\mathrm{i}z+z_0)\mathrm{d}z\\
		&=
		\int_{-\infty-\mathrm{i}(t+z_0)}^{\infty-\mathrm{i}(t+z_0)}
		\mathrm{e}^{-\varepsilon z^2}
		\hat{\mathcal{A}}(z+\mathrm{i}z_0)f(\mathrm{i}z)\mathrm{d}z\\
		&=
		\int_{-\infty-\mathrm{i}t}^{\infty-\mathrm{i}t}
		\mathrm{e}^{-\varepsilon z^2}
		\hat{\mathcal{A}}(z+\mathrm{i}z_0)f(\mathrm{i}z)\mathrm{d}z,
	\end{align*}
	and we still have
	\begin{enumerate}
		\item $\lim\limits_{\varepsilon\to 0}\tilde{g}_{\varepsilon}(z_0)=f(A+z_0)$;
		\item Sequence $(\tilde{g}_\varepsilon)_{\varepsilon>0}$ is uniformly bounded.
	\end{enumerate}
	By the Leibniz integral rule again, we have
	\begin{equation*}
		\tilde{g}'_{\varepsilon}(z_0)=
		\frac{1}{\sqrt{2\pi}}
		\int_{-\infty-\mathrm{i}t}^{\infty-\mathrm{i}t}
		\mathrm{e}^{-\varepsilon z^2}
		\mathrm{i}\hat{\mathcal{A}}'(z+\mathrm{i}z_0)
		f(\mathrm{i}z)\mathrm{d}z,
	\end{equation*}
	this completes the proof.
\end{proof}
\noindent
This method seems to be able to replace $\mathrm{D}$ with an umbrae whose dominating interval is also $\mathbb{R}$. In fact, Theorem \ref{naturalthem} will strengthen the condition to prove more general cases.

\subsection{Specific calculation}

After establishing the required well-definedness, we can see how various conclusions emerge naturally.
\begin{them}
	\label{naturalthem}
	Suppose $(f(\mathrm{i}z),\Omega_{\beta,\alpha})$ is a umbrae with index $(b,a)$.

	If $A_k=(\mathcal{A}_k,\Omega_{a_k,b_k})$ is a umbrae with index $(\alpha_k, \beta_k)$, satisfied
	\begin{enumerate}
		\item $a_k<a$; $b<b_k$;
		\item $\beta_k-\alpha_k>-(\alpha-\beta)$,
	\end{enumerate}
	then we have
	\begin{enumerate}
		\item $(f(A_k+\mathrm{i}z),\Omega_{\beta-\beta_k,\alpha-\alpha_k})$ is a umbrae with index at least $(b\vee a_k,a\wedge b_k)$;
		\item If $(a_k,b_k)\cap(a_j,b_j)\neq\varnothing$, and $(\beta_k-\alpha_k)+(\beta_j-\alpha_j)>-(\alpha-\beta)$ then $f(A_k+A_j+\mathrm{i}z)$ is well-defined in the sense of Definition \ref{welldefplus};
		\item If $(a_k,b_k)\cap(a_j,b_j)\neq\varnothing$, and $(\beta-\beta_k,\alpha-\alpha_k)\cap(\beta-\beta_j,\alpha-\alpha_j)\neq\varnothing$ then $f((A_k[+]A_j)+\mathrm{i}z)$ is well-defined, satisfied
		\begin{equation}
			f((A_k[+]A_j)+\mathrm{i}z)=
			f(A_k+\mathrm{i}z)+f(A_j+\mathrm{i}z).
		\end{equation}
	\end{enumerate}
\end{them}
\begin{proof}
	First assume that $\alpha_k<\beta_k$, and take $a_k<\tilde{a}_k<\tilde{a}<a, b<\tilde{b}<\tilde{b}_k<b_k$, we have
	\begin{align*}
		|\hat{\mathcal{A}}(z)|
		&\leq 
		\begin{cases*}
			C_1(-\mathrm{Im} z)\cdot
			\mathrm{e}^{-\tilde{b}_k\mathrm{Re} z}
			&$\mathrm{Re} z\geq 0, 
			\alpha_k<-\mathrm{Im} z<\beta_k$\\
			C_1(-\mathrm{Im} z)\cdot
			\mathrm{e}^{-\tilde{a}_k\mathrm{Re} z}
			&$\mathrm{Re} z\leq 0, 
			\alpha_k<-\mathrm{Im} z<\beta_k$
		\end{cases*}\\
		|f(\mathrm{i}z)|
		&\leq 
		\begin{cases*}
			C_2(-\mathrm{Im} z)\cdot
			\mathrm{e}^{\tilde{b}\mathrm{Re} z}
			&$\mathrm{Re} z\geq 0, 
			\beta<-\mathrm{Im} z<\alpha$\\
			C_2(-\mathrm{Im} z)\cdot
			\mathrm{e}^{\tilde{a}\mathrm{Re} z}
			&$\mathrm{Re} z\leq 0, 
			\beta<-\mathrm{Im} z<\alpha$
		\end{cases*}
	\end{align*}
	So we can directly verify that, for $\mathrm{Re} z_0\in(\beta-\beta_k,\alpha-\alpha_k)$ there exists $t$ such that
	\begin{align*}
		|f(A+z_0)|&\leq
		\int_{-\infty-\mathrm{i}t}^{\infty-\mathrm{i}t}
		|\hat{\mathcal{A}}(z)|\cdot
		|f(\mathrm{i}z+z_0)|\mathrm{d}z\\
		&\lesssim
		\mathrm{e}^{\tilde{b}\mathrm{Im} z_0}
		\int_{0}^{\infty}
		\mathrm{e}^{-(\tilde{b}_k-\tilde{b})x}\mathrm{d}x+
		\mathrm{e}^{\tilde{b}\mathrm{Im} z_0}
		\int_{-\mathrm{Im}z_0}^{0}
		\mathrm{e}^{(\tilde{b}-\tilde{a}_k)x}\mathrm{d}x+\\
		&\phantom{==}
		\mathrm{e}^{\tilde{a}\mathrm{Im} z_0}
		\int_{-\infty}^{-\mathrm{Im}z_0}
		\mathrm{e}^{(\tilde{a}-\tilde{a}_k)x}\mathrm{d}x\\
		&\leq
		\frac{\mathrm{e}^{\tilde{b}\mathrm{Im} z_0}}{\tilde{b}_k-\tilde{b}}+
		\frac{\mathrm{e}^{\tilde{b}\mathrm{Im} z_0}-\mathrm{e}^{\tilde{a}_k\mathrm{Im} z_0}}{\tilde{b}-\tilde{a}_k}+
		\frac{\mathrm{e}^{\tilde{a}_k\mathrm{Im} z_0}}{\tilde{a}-\tilde{a}_k}
		\quad \mathrm{Im}z_0\geq 0.
	\end{align*}
	Repeat the same method to estimate the other half and we can conclude that $(f(A+\mathrm{i}z),\Omega_{\beta-\beta_k,\alpha-\alpha_k})$ is a umbrae with index $(b\vee a_k,a\wedge b_k)$. Finally, apply Theorem \ref{decompum} and Definition \ref{singdef} to $A$.
\end{proof}
\noindent
The condition of the umbrae under Theorem \ref{naturalthem} seems to be the most natural, but not enough for our purposes. We can also notice that the performance of the umbrae is more determined by its dominating interval.

In order to be able to perform basic calculations, we need the following theorem.
\begin{lemm}
	\label{rapiddec}
	If $(f(\mathrm{i}z),\Omega_{\beta,\alpha})$ is a regular umbrae with index $(b,a)$ satisfied $0\in(b,a)$, then for every $s\in(\beta,\alpha)$ we have $f(s+\mathrm{i}\xi)\in\mathcal{S}(\mathbb{R})$.
\end{lemm}
\begin{proof}
	Apply Theorem \ref{corft}.
\end{proof}
\begin{lemm}
	Suppose $A=(\mathcal{A},\Omega_{a,b})$ is a regular umbrae with index $(\alpha,\beta)$, and
	\begin{enumerate}
		\item $-\infty<\alpha<\beta$;
		\item For every $t\in(a,b)$ we have $\mathrm{e}^{-\alpha x}\mathcal{A}(x-\mathrm{i}t)\in\mathcal{S}'(\mathbb{R})$;
		\item Denote $h(\xi)\coloneqq
		\mathrm{e}^{-t\xi}
		\mathcal{F}(\mathrm{e}^{-\alpha (x-\mathrm{i}t)}\mathcal{A}(x-\mathrm{i}t))$ in the sense of $\mathcal{S}'(\mathbb{R})$.
	\end{enumerate}
	If $(f(\mathrm{i}z),\Omega_{\beta_0,\alpha_0})$ is a umbrae with index $(b_0,a_0)$ satisfied
	\begin{equation*}
		a<a_0,\quad b_0<b,
	\end{equation*}
	then for every $s\in(\beta_0,\alpha_0)$ we have
	\begin{align*}
		f(A+\alpha-s)
		&=
		\int_{\mathbb{R}}
		\mathrm{e}^{t\xi}
		h(\xi)
		\mathrm{e}^{-t\xi}
		f(s+\mathrm{i}\xi)\mathrm{d}\xi.
	\end{align*}
\end{lemm}
\begin{proof}
	According to the continuity of Fourier transform on $\mathcal{S}'(\mathbb{R})$, for every $t\in(a,b)$ we have 
	\begin{equation}
		\lim\limits_{\varepsilon\to 0^+}
		\mathrm{e}^{t\xi}
		\hat{\mathcal{A}}(\xi-\mathrm{i}(\alpha+\varepsilon))=
		\mathrm{e}^{t\xi}h(\xi),
	\end{equation}
	in the sense of $\mathcal{S}'(\mathbb{R})$.
	
	Apply Theorem \ref{decompum} on $f$ to get
	\begin{enumerate}
		\item $(f^+(\mathrm{i}z),\Omega_{\beta_0,\alpha_0})$ is a regular umbrae with index $(b_0,\infty)$;
		\item $(f^-(\mathrm{i}z),\Omega_{\beta_0,\alpha_0})$ is a regular umbrae with index $(-\infty,a_0)$;
		\item $f=f^+-f^-$,
	\end{enumerate}
	according to Lemma \ref{rapiddec}, for every $s\in(\beta_0,\alpha_0)$ we have
	\begin{align*}
		\int_{\mathbb{R}}
		\hat{\mathcal{A}}(\xi-\mathrm{i}(\alpha+\varepsilon))
		f^+(s+\mathrm{i}\xi)\mathrm{d}\xi
		&=
		\int_{\mathbb{R}}
		\mathrm{e}^{\tilde{b}\xi}
		\hat{\mathcal{A}}(\xi-\mathrm{i}(\alpha+\varepsilon))
		\mathrm{e}^{-\tilde{b}\xi}
		f^+(s+\mathrm{i}\xi)\mathrm{d}\xi\\
		&\to
		\int_{\mathbb{R}}
		\mathrm{e}^{\tilde{b}\xi}
		h(\xi)
		\mathrm{e}^{-\tilde{b}\xi}
		f^+(s+\mathrm{i}\xi)\mathrm{d}\xi\\
		&=
		\int_{\mathbb{R}}
		\mathrm{e}^{t\xi}
		h(\xi)
		\mathrm{e}^{-t\xi}
		f^+(s+\mathrm{i}\xi)\mathrm{d}\xi,
	\end{align*}
	Repeat the same method to the other half and we can conclude that
	\begin{align*}
		\int_{\mathbb{R}}
		\hat{\mathcal{A}}(\xi-\mathrm{i}(\alpha+\varepsilon))
		f(s+\mathrm{i}\xi)\mathrm{d}\xi
		&\to
		\int_{\mathbb{R}}
		\mathrm{e}^{t\xi}
		h(\xi)
		\mathrm{e}^{-t\xi}
		f(s+\mathrm{i}\xi)\mathrm{d}\xi\\
		f(A+(\alpha+\varepsilon)-s)
		&\to
		f(A+\alpha-s),
	\end{align*}
	where $a<\tilde{a}<a_0, b_0<\tilde{b}<b, a<t<b$. The continuity of $f(A+z)$ comes from Theorem \ref{naturalthem}.
\end{proof}
\begin{them}
	\label{calu}
	Suppose $A=(\mathcal{A},\Omega_{a,b})$ is a umbrae with index $(\alpha,\beta)$, and
	\begin{enumerate}
		\item $-\infty<\alpha\leq\beta$;
		\item For every $t\in(a,b)$ we have $\mathrm{e}^{-\alpha x}\mathcal{A}(x-\mathrm{i}t)\in\mathcal{S}'(\mathbb{R})$;
		\item Denote $h(\xi)\coloneqq
		\mathrm{e}^{-t\xi}
		\mathcal{F}(\mathrm{e}^{-\alpha 	(x-\mathrm{i}t)}\mathcal{A}(x-\mathrm{i}t))$ in the sense of $\mathcal{S}'(\mathbb{R})$,
	\end{enumerate}
	If $(f(\mathrm{i}z),\Omega_{\beta_0,\alpha_0})$ is a umbrae with index $(b_0,a_0)$ satisfied
	\begin{enumerate}
		\item $a<a_0$; $b_0<b$;
		\item $\alpha\in(\beta_0,\alpha_0)$,
	\end{enumerate}
	then we have
	\begin{align}
		f(A)
		&=
		\int_{\mathbb{R}}
		\mathrm{e}^{t\xi}
		h(\xi)
		\mathrm{e}^{-t\xi}
		f(\alpha+\mathrm{i}\xi)\mathrm{d}\xi.
	\end{align}
\end{them}
\begin{proof}
	For the singular umbrae $A$, only need to apply Theorem \ref{decompum}.
\end{proof}
\begin{corr}
	\label{specficcalu}
	Suppose $f$ is an exponential type analytic function defined on $\{z\in\mathbb{C}:\mathrm{Re}z\in (\beta,\alpha)\}$.
	\begin{enumerate}
		\item $f((c)+z)=f(c+z)$;
		\item $f(\mathrm{D}+z)=f'(z)$;
		\item $f(\Delta+z)=f(z+1)-f(z)$, where $\beta-\alpha>1$.
	\end{enumerate}
\end{corr}
\begin{prop}
	Suppose $A=(\mathcal{A},\Omega_{a,b})$ is a umbrae with index $(\alpha,\beta)$.
	\begin{enumerate}
		\item $\mathrm{e}^{Az}=\mathcal{A}(z)\quad z\in\Omega_{a,b}$;
		\item $A^n=\mathcal{A}^{(n)}(0)\quad 0\in(a,b),n\in\mathbb{N}$;
		\item $(A^{\mathrm{i}z},\Omega_{-\infty,\infty})$ is a regular umbrae with index $(-\pi,\pi)$, where $0\in(a,b),\beta>0$.
	\end{enumerate}
\end{prop}
\begin{proof}
	For the singular umbrae $A$, only need to apply Theorem \ref{decompum}.
\end{proof}
\noindent
Just like Carlson's theorem shows, the function class formed by such $A^z$ is also determined by the value on natural numbers.

\subsection{Limit calculus}

Next, we need several important lemmas to help us calculate $f(A)$ specifically.
\begin{lemm}
	\label{convlem}
	Suppose $f$ is locally integrable, $b>0,n\geq 0$.
	
	If $\mathrm{e}^{-bx}(1+|x|)^n f(x)\in L^1(0,\infty)$, then $\mathrm{e}^{-bx}(1+|x|)^n \int_{0}^{x}f(t)\mathrm{d}t\in L^1(0,\infty)$.
\end{lemm}
\begin{proof}
	Only need to notice that
	\begin{align*}
		\mathrm{e}^{-bx}(1+x)^n\int_{0}^{x}|f(t)|\mathrm{d}t
		&\leq
		\int_{0}^{x}\mathrm{e}^{-bt}(1+t)^n|f(t)|
		(1+x-t)^n\mathrm{e}^{-b(x-t)}\mathrm{d}t\\
		&=
		(\mathrm{e}^{-b(\cdot)}(1+\cdot)^n f(\cdot)
		\ast\mathrm{e}^{-b(\cdot)}(1+\cdot)^n)(x).
	\end{align*}
	The result is obtained from Young's inequality.
\end{proof}
The umbrae that satisfy the following properties are good enough to give a more accurate estimate for $\hat{\mathcal{A}}$, and Bernoulli umbrae $\mathrm{B}$ is such a umbrae. Ultimately, it will allow us to do calculus $f(A)$ on as many functions as possible. In fact, if there is no such fine result, we will not be able to process the Gosper series as we promised.
\begin{lemm}[Good umbrae]
	\label{goodumbrae}
	Suppose $A=(\mathcal{A},\Omega_{a,b})$ is a regular umbrae with index $(\alpha,\beta)$.
	\begin{enumerate}
		\item If $b<\infty$ and $(\mathcal{A}(z)
		\prod_{j=1}^{m}(z-b_j)^{l_j},\Omega_{a,b+\varepsilon})$ is a  umbrae, then we have
		\begin{align}
			\label{cond+}
			|\hat{\mathcal{A}}(\xi-\mathrm{i}t)|\leq
			C(t)\mathrm{e}^{-b\xi}(1+|\xi|)^{l-1}
			\quad \xi\geq 0,
		\end{align}
		where $l=\max\{l_1,\ldots,l_m\},-\mathrm{Im} b_j=b$;
		\item If $a>-\infty$ and $(\mathcal{A}(z)
		\prod_{j=1}^{n}(z-a_j)^{k_j},\Omega_{a-\varepsilon,b})$ is a  umbrae, then we have
		\begin{align}
			\label{cond-}
			|\hat{\mathcal{A}}(\xi-\mathrm{i}t)|\leq
			C(t)\mathrm{e}^{-a\xi}(1+|\xi|)^{k-1}
			\quad \xi\leq 0,
		\end{align}
		where $k=\max\{k_1,\ldots,k_n\},-\mathrm{Im} a_j=a$.
	\end{enumerate}
\end{lemm}
\begin{proof}
	Consider the Fourier transform of the singular part of $\hat{\mathcal{A}}$ at the pole
\end{proof}
\noindent
It seems that this lemma can be generalized further, but the current form is sufficient for us.
\begin{them}[Limit calculus]
	\label{Limitcalculus}
	Suppose $A$ is a regular umbrae, $(f_n)_{n\in\mathbb{N}}$ is a sequence in $\mathcal{T}_A(I)$, and $f\in\mathcal{T}_A(I), t\in I$.
	
	If $\hat{\mathcal{A}}(z)f_n(\mathrm{i}z)$ tends to $\hat{\mathcal{A}}(z)f(\mathrm{i}z)$ in the sense of $L^1(\mathbb{R}-\mathrm{i}t)$, then
	\begin{equation}
		\lim\limits_{n\to\infty}f_n(A)=f(A).
	\end{equation}
\end{them}
\begin{corr}[Interchangeability]
	\label{Interchangeability}
	Suppose $A=(\mathcal{A},\Omega_{a,b})$ is a umbrae with index $(\alpha,\beta)$, and satisfy $0\in(a,b)$.
	
	For $f(z)=\sum_{n=0}^{\infty}\frac{a_n z^n}{n!}$ we have
	\begin{enumerate}
		\item If there exists $r<(-a)\vee b$ such that
		\begin{equation}
			\sum_{n=0}^{\infty}|a_n| r^{-n}
			< \infty,
		\end{equation}
		then $f(A)=\sum_{n=0}^{\infty}\frac{a_n A^n}{n!}$;
		\item If $A$ satisfied one of the conditions in Lemma \ref{goodumbrae}, and $f$ satisfies
		\begin{align*}
			\sum_{n=0}^{\infty}|b|^{-n}n^{l-1}|a_n|
			&< \infty\\
			\sum_{n=0}^{\infty}|a|^{-n}n^{k-1}|a_n|
			&< \infty,
		\end{align*}
		then $f(A)=\sum_{n=0}^{\infty}\frac{a_n A^n}{n!}$.
	\end{enumerate}
\end{corr}
\begin{proof}
	For the singular umbrae $A$, only need to apply Theorem \ref{decompum}.
\end{proof}
\noindent
This corollary will eventually provide us the interchangeable results which we care about.
\begin{defi}[Hierarchy]
	\label{hierarchydef}
	If $f$ is defined on $\{z\in\mathbb{C}:\mathrm{Re} z>0\}$ and satisfies the following properties
	\begin{enumerate}
		\item For $\mathrm{Re} t>0$ we have
		\begin{align*}
			|f(t+\mathrm{i}\xi)|
			\mathrm{e}^{-b\xi}(1+|\xi|)^{l}
			&\in
			L^1((0,\infty);\xi)\\
			\lim\limits_{\mathrm{Re}t\to\infty}
			\int_{-\infty}^{\infty}
			\left |
			f(t+\mathrm{i}\xi)-
			\sum_{k=0}^{p-1}\frac{f^{(k)}(t)}{k!}(\mathrm{i}\xi)^k
			\right |
			\mathrm{e}^{-b\xi}(1+|\xi|)^{l}
			\mathrm{d}\xi
			&=0,
		\end{align*}
		locally uniformly with respect to $\mathrm{Im}t$;
		\item For $\mathrm{Re} t>0$ we have
		\begin{align*}
			|f(t+\mathrm{i}\xi)|
			\mathrm{e}^{-a\xi}(1+|\xi|)^{k}
			&\in
			L^1((-\infty,0);\xi)\\
			\lim\limits_{\mathrm{Re}t\to\infty}
			\int_{-\infty}^{\infty}
			\left |
			f(t+\mathrm{i}\xi)-
			\sum_{k=0}^{p-1}\frac{f^{(k)}(t)}{k!}(\mathrm{i}\xi)^k
			\right |
			\mathrm{e}^{-a\xi}(1+|\xi|)^{k}
			\mathrm{d}\xi
			&=0,
		\end{align*}
		locally uniformly with respect to $\mathrm{Im}t$;
		\item $|f(z)|\leq C(\mathrm{Re}z)\cdot\mathrm{e}^{s|\mathrm{Im}z|}$, where $C$ is locally bounded,
	\end{enumerate}
	then the space formed by such functions is denoted by $\mathcal{T}_{a,k;b,l}^{(p-1)}$.
\end{defi}
\noindent
In fact, we only need $t$ to go to infinity along the natural number.
\begin{them}[Hierarchy]
	\label{hierarchy}
	If $0\in(a,b)$ and $k,l\geq 0$, then
	\begin{equation}
		f(z)\in\mathcal{T}_{a,k;b,l}^{(p-1)}
		\Rightarrow
		\int_{1}^{x}f(t)\mathrm{d}t\in\mathcal{T}_{a,k;b,l}^{(p)}.
	\end{equation}
\end{them}
\begin{proof}
	Denote $F(z)=\int_{1}^{x}f(t)\mathrm{d}t$. Notice that
	\begin{align*}
		\left\lVert 
		\mathrm{e}^{-b\xi}(1+|\xi|)^{l}
		\left(
		F(t+\mathrm{i}\xi)-
		\sum_{k=0}^{p}\frac{F^{(k)}(t)}{k!}(\mathrm{i}\xi)^k
		\right)
		\right\rVert_{L^1(0,\infty)}
		&=\\
		\left\lVert 
		\mathrm{e}^{-b\xi}(1+|\xi|)^{l}
		\int_{0}^{\xi}
		\left(
		f(t+\mathrm{i}\xi_0)-
		\sum_{k=0}^{p-1}\frac{f^{(k)}(t)}{k!}(\mathrm{i}\xi_0)^k
		\right)\mathrm{d}\xi_0
		\right\rVert_{L^1(0,\infty)}
		&\leq\\
		\left\lVert
		\mathrm{e}^{-b\xi}(1+|\xi|)^{l}
		\right\rVert_{L^1(0,\infty)}
		\cdot
		\left\lVert 
		\mathrm{e}^{-b\xi}(1+|\xi|)^{l}
		\left(
		f(t+\mathrm{i}\xi)-
		\sum_{k=0}^{p-1}\frac{f^{(k)}(t)}{k!}(\mathrm{i}\xi)^k
		\right)
		\right\rVert_{L^1(0,\infty)}
		&\to 0,
	\end{align*}
	Repeat the same method to estimate the other half and we can conclude that \eqref{hierarchy} holds.
\end{proof}
\noindent
Similar to Theorem \ref{calu}, this theorem will help us calculate $f(\mathrm{B})$ in a more direct way.

\section{Application}

Let’s first redeem the promise and see how the MS-fractional summation related to the Bernoulli umbrae.

Although the selection of $\Omega_{a,b}$ does not need to include the origin, here we still only care about the Bernoulli umbrae that includes the origin, which allows us to define $\mathrm{B}^n$. Some conclusions can be generalized to other Bernoulli umbras.
\begin{lemm}
	$\mathrm{B}+\Delta=1+\mathrm{D}.$
\end{lemm}
\begin{them}
	\label{MScomp}
	If there exists $p\in\mathbb{N}$ such that
	\begin{enumerate}
		\item $f\in\mathcal{T}_{-2\pi,0;2\pi,0}^{(p-1)}$;
		\item $\lim\limits_{t\to\infty}f^{(p)}(t+\mathrm{i}\xi)=0$ locally uniformly,
	\end{enumerate}
	then for every $\mathrm{Re} z>-1$ we have
	\begin{equation}
		(\mathrm{MS})\text{-}
		\sum_{k=1}^{z}f(k)=
		\int_{\mathrm{B}}^{\mathrm{B}+z}f(k)\mathrm{d}k.
	\end{equation}
\end{them}
\begin{proof}
	Denote $F(z)=\int_{1}^{x}f(t)\mathrm{d}t$. Notice that the Bernoulli umbrae $\mathrm{B}$ satisfies the condition of Lemma \ref{goodumbrae}, therefore according to Definition \ref{Evaldef} and \ref{hierarchydef}, $F(t+\mathrm{B})$ is defined for $\mathrm{Re} t>0$.
	
	By Theorem \ref{hierarchy}, we have
	\begin{align*}
		\lim\limits_{n\to\infty}
		\left(
		F(n+\mathrm{B})-
		\sum_{k=0}^{p}
		\frac{f^{(k-1)}(n)\mathrm{B}_k}{k!}
		\right)=0.
	\end{align*}
	By the Lagrangian remainder theorem, we have
	\begin{equation*}
		\sum_{k=0}^{p}
		\frac{(f^{(k-1)}(t+y)-f^{(k-1)}(t))\mathrm{B}_k}{k!}
		=
		\sum_{k=0}^{p}
		a_k(t)y^k+o(1)y^{p+1},
	\end{equation*}
	therefore we have polynomial $P_n(y)=\sum_{k=0}^{p}a_k(n)y^k$ such that
	\begin{align*}
		\lim\limits_{n\to\infty}
		\left(
		\sum_{k=1}^{n}
		(f(k+y)-f(k))-P_n(y)
		\right)
		&=\\
		\lim\limits_{n\to\infty}
		(
		(F(n+y+\mathrm{B})-F(n+\mathrm{B}))
		-P_n(y))-
		\int_{\mathrm{B}}^{\mathrm{B}+y}
		f(k)\mathrm{d}k
		&=\\
		\lim\limits_{n\to\infty}
		o(1)y^{p+1}-
		\int_{\mathrm{B}}^{\mathrm{B}+y}
		f(k)\mathrm{d}k
		&=
		-\int_{\mathrm{B}}^{\mathrm{B}+y}
		f(k)\mathrm{d}k.
	\end{align*}
	Finally, by Definition \ref{singdef} and Corollary \ref{specficcalu}, we have
	\begin{align*}
		\mathrm{B}+\Delta=1+\mathrm{D}
		&\Rightarrow
		(\mathrm{B}+1)[-]\mathrm{B}=1+\mathrm{D}\\
		&\Rightarrow
		F(y+((\mathrm{B}+1)[-]\mathrm{B}))=F(y+1+\mathrm{D})\\
		&\Rightarrow
		F(y+\mathrm{B}+1)-F(t+\mathrm{B})=F'(y+1),
	\end{align*}
	which satisfies the conditions in Definition \ref{MSdef}.
\end{proof}
\begin{corr}
	\label{MSinter}
	If \begin{equation*}
		f(z)=
		\sum_{m=1}^{M}b_m z^{-m}+
		\sum_{n=0}^{\infty}\frac{a_n}{n!} z^n
	\end{equation*}
	and there exists $p\in\mathbb{N}$ satisfies that
	\begin{enumerate}
		\item $f\in\mathcal{T}_{-2\pi,0;2\pi,0}^{(p-1)}$;
		\item $\lim\limits_{t\to\infty}f^{(p)}(t+\mathrm{i}\xi)=0$ locally uniformly;
		\item $\sum_{n=0}^{\infty}(2\pi)^{-n}|a_n|<\infty$,
	\end{enumerate}
	then for every $\mathrm{Re}\alpha,\mathrm{Re}\beta>-1$ we have
	\begin{equation}
		(\mathrm{MS})\text{-}
		\sum_{k=\alpha+1}^{\beta}f(k)=
		\sum_{m=1}^{M} 
		\left(
		(\mathrm{MS})\text{-}
		\sum_{k=\alpha+1}^{\beta}b_mz^{-m}
		\right)+
		\sum_{n=0}^{\infty}
		\left(
		(\mathrm{MS})\text{-}
		\sum_{k=\alpha+1}^{\beta}\frac{a_n}{n!}z^n
		\right).
	\end{equation}
\end{corr}
\begin{proof}
	Apply Corollary \ref{Interchangeability} and Theorem \ref{MScomp}.
\end{proof}
\begin{corr}
	\label{MSinterderi}
	If there exists $p\in\mathbb{N}$ satisfies that
	\begin{enumerate}
		\item $f'\in\mathcal{T}_{-2\pi,0;2\pi,0}^{(p-1)}$;
		\item $\lim\limits_{t\to\infty}f^{(p+1)}(t+\mathrm{i}\xi)=0$ locally uniformly;
		\item $\int_{-\infty}^{\infty}
		\mathrm{e}^{-2\pi|\xi|}
		|f(t+\mathrm{i}\xi)|\mathrm{d}\xi$ is locally bounded for $\mathrm{Re}t>0$,
	\end{enumerate}
	then we have
	\begin{equation*}
		(\mathrm{MS})\text{-}
		\frac{\mathrm{d}}{\mathrm{d}z}
		\sum_{k=1}^{z}f(k)=
		f(\mathrm{B})+
		(\mathrm{MS})\text{-}
		\sum_{k=1}^{z}
		\frac{\mathrm{d}}{\mathrm{d}k}
		f(k).
	\end{equation*}
\end{corr}
\begin{proof}
	Apply Theorem \ref{intercderi} and Lemma \ref{convlem}.
\end{proof}
\noindent
In fact, the condition of $f$ in Corollary \ref{MSinter} also satisfies the third condition in Corollary \ref{MSinterderi}.

At this point, we justify the speculation from  Müller and Schleicher. Next, we will see that these conditions are sufficient to cover the special series they care about.

Finally, it needs to be emphasized that the processing of Bernoulli umbrae is essentially equivalent to applying the Abel-Plana formula. Although we have removed the decay condition of the function when the imaginary part is large, this is actually done by requiring the function to be exponential growth.

\subsection{Gosper series}

Let's consider a sufficiently general example.

Suppose
\begin{align}
	J(z)
	&=
	\sum_{n=0}^{\infty}
	\frac{a_{2n}}{(2n)!}z^{2n}
\end{align}
satisfies that
\begin{enumerate}
	\item $|a_n|\leq C\cdot (1+n)^{\nu_a}$;
	\item $|J(z)|\leq C\cdot
	\mathrm{e}^{|\mathrm{Im} z|}
	(1+|z|)^{\nu_J}\quad\mathrm{Re}z>0$.
\end{enumerate}
\begin{lemm}~ 
	\begin{enumerate}
		\item $|\mathrm{Im}\sqrt{z_1+z_2}|
		\leq
		|\mathrm{Im}\sqrt{z_1}|+|\mathrm{Im}\sqrt{z_2}|$;
		\item $|\tilde{J}(z)|\leq
		C\cdot\mathrm{e}^{|z|}(1+|z|)^{\nu_a}$, where $\tilde{J}(z)=
		\sum_{n=0}^{\infty}
		\frac{|a_{2n}|}{(2n)!}z^{2n}$.
	\end{enumerate}
\end{lemm}
\begin{prop}
	\label{theyareright}
	$z^{-n}J(\sqrt{(2\pi z)^2+b^2})$ satisfies all the conditions of Corollary \ref{MSinter}, where $n>\max\{\nu_a,\nu_J\}+1$.
\end{prop}
\begin{proof}
	Suppose that
	\begin{equation*}
		z^{-n}J(\sqrt{z^2+b^2})=
		\sum_{m=1}^{M}b_m z^{-m}+
		\sum_{n=0}^{\infty}\frac{c_n}{n!} z^n,
	\end{equation*}
	only need to notice that
	\begin{align*}
		\int_{1}^{\infty}
		x^{-n}\tilde{J}(\sqrt{x^2+|b|^2})
		\mathrm{e}^{-x}\mathrm{d}x
		<\infty
		\Rightarrow
		\sum_{n=0}^{\infty}|c_n|<\infty,
	\end{align*}
	and for $n>\nu_a+1$ we have
	\begin{align*}
		\int_{1}^{\infty}
		x^{-n}\tilde{J}(\sqrt{x^2+|b|^2})
		\mathrm{e}^{-x}\mathrm{d}x
		\leq C\cdot
		\int_{1}^{\infty}x^{\nu_a-n}
		\mathrm{e}^{\sqrt{x^2+|b|^2}-x}\mathrm{d}x<\infty.
	\end{align*}
	When the real part of $z$ is large enough, we have
	\begin{align*}
		\mathrm{e}^{-|\mathrm{Im}z|}
		|z^{-n}J(\sqrt{z^2+b^2})|
		&\leq
		C\cdot
		\mathrm{e}^{|\mathrm{Im}\sqrt{z^2+b^2}|
		-|\mathrm{Im}z|}
		|z|^{\nu_J-n}\\
		&\leq
		C\cdot|z|^{\nu_J-n},
	\end{align*}
	and for $n>\nu_J+1$ we have $z^{-n}J(\sqrt{z^2+b^2})\in\mathcal{T}_{-1,0;1,0}^{(-1)}$.
\end{proof}
For example, we can take
\begin{equation}
	J(z)=z^{-\nu}J_{\nu}(z)=
	\sum_{n=0}^{\infty}
	\frac{(-1)^nz^{2n}}{2^{\nu+2n}n!(n+\nu)!}.
\end{equation}
From the Poisson integral representation, for $\mathrm{Re}\nu>-\frac{1}{2}$ we have
\begin{align*}
	J(z)&=
	\frac{1}{2^\nu (\nu-\frac{1}{2})!\sqrt{\pi}}
	\int_{0}^{\pi}
	\mathrm{e}^{\mathrm{i}z\cos\theta}
	\sin^{2\nu}\theta\mathrm{d}\theta\\
	|J(z)|&\leq C\cdot \mathrm{e}^{|\mathrm{Im}z|},
\end{align*}
therefore $\nu_a=-(\mathrm{Re}\nu+\frac{1}{2}),\nu_J=0$. In particular, when $\nu=1/2$ we have $J(z)=\sqrt{\frac{2}{\pi}}\frac{\sin z}{z},n>0$, when $\nu=-1/2$ we have $J(z)=\sqrt{\frac{2}{\pi}}\cos z,n>1$.

We can improve the related results in \cite{Gosper1993}, that is, generalize $b$ to every complex number.
\begin{corr}[Gosper series]~ 
	\label{Gosperseries}
	If $J(z)=\sum_{n=0}^{\infty}
	\frac{a_{2n}}{(2n)!}z^{2n}$ satisfies that
	\begin{enumerate}
		\item $|a_n|\leq C\cdot (1+n)^{\nu_a}$;
		\item $|J(z)|\leq C\cdot
		\mathrm{e}^{|\mathrm{Im} z|}
		(1+|z|)^{\nu_J}\quad\mathrm{Re}z>0$,
	\end{enumerate}
	and $b\in\mathbb{C}$, then
	\begin{enumerate}
		\item For $\nu_a,\nu_J<0$ we have
		\begin{align}
			(\mathrm{MS})\text{-}
			\sum_{n=1/4}^{-1/4}
			\frac{J(\sqrt{b^2+(2\pi n)^2})}{n}
			&=\pi J(b)\\
			\sum_{n=0}^{\infty}
			\frac{(-1)^n}{n+\frac{1}{2}}J(\sqrt{b^2+\pi^2(n+1/2)^2})
			&=\frac{\pi}{2} J(b);
		\end{align}
		\item For $\nu_a,\nu_J<1$ we have
		\begin{align}
			(\mathrm{MS})\text{-}
			\sum_{n=1}^{-1/2}
			\frac{J(\sqrt{b^2+(2\pi n)^2})}{n^2}
			&=-\frac{\pi^2J(b)}{3}-\frac{\pi^2J'(b)}{b}\\
			\sum_{n=0}^{\infty}
			\frac{(-1)^n}{n^2}J(\sqrt{b^2+\pi^2n^2})
			&=-\frac{\pi^2J(b)}{12}-\frac{\pi^2J'(b)}{4b}.
		\end{align}
	\end{enumerate}
\end{corr}
\begin{proof}
	By Proposition \ref{theyareright}, we can apply Corollary \ref{MSinter} on $z^{-n}J(\sqrt{(2\pi z)^2+b^2})$.
\end{proof}
\noindent
By applying $|f(a)-f(b)|\leq M|a-b|$, we are able to generalize further to the situation where $\sqrt{z^2+b^2}$ is replaced by $\sqrt[n]{z^n+p(z)}$.

The common processing on $\sum_{n\in\mathbb{Z}}(-1)^n f(n)$ or $\sum_{n\in\mathbb{Z}}f(n)$ is to apply the Abel-Plana formula, which is essentially equivalent to the processing here.

\subsection{Bernoulli umbrae}

This subsection will mainly introduce how to use Bernoulli umbrae $\mathrm{B}$ to synthesize some common analysis results. This approach can also be easily extended to similar umbrae such as Euler umbrae $\mathrm{E}$.

The proof of the following theorem can actually be given directly by its corollaries, so it is omitted here.
\begin{them}[Bernoulli umbrae]~ 
	\begin{enumerate}
		\item $\ln\mathrm{B}=-\gamma$;
		\item $\mathrm{B}\ln\mathrm{B}=\ln\sqrt{\frac{\mathrm{e}}{2\pi}}$;
		\item $\mathrm{B}^2\ln\mathrm{B}=\frac{1}{4}-2\ln A$, where $A$ is Glaisher–Kinkelin constant;
		\item $\mathrm{B}^z=-z\zeta(1-z)$.
	\end{enumerate}
\end{them}
\begin{lemm}[Multiplication theorem]
	\label{Multiplicationtheorem}
	\[
	n\mathrm{B}[+](n\mathrm{B}-1)[+]\cdots[+](n\mathrm{B}-n+1)=[n]+\mathrm{B}.
	\]
\end{lemm}
\noindent
In fact, this lemma synthesized all the multiplication theorems about special functions.
\begin{lemm}
	$\left(
	\frac{\mathrm{d}}{\mathrm{d}z}
	\right)^n
	\mathrm{B}^z=\mathrm{B}^z\ln^n\mathrm{B}$.
\end{lemm}
\begin{proof}
	Apply the argument similar to Theorem \ref{intercderi}, note that
	\begin{equation*}
		\int_{-\infty-\mathrm{i}t}^{\infty-\mathrm{i}t}
		|\hat{\mathcal{A}}(z)|\cdot
		|(\mathrm{i}z)^{z_0}\ln^n(\mathrm{i}z)|\mathrm{d}z
	\end{equation*}
	is locally uniformly bounded for $z_0$.
\end{proof}
The previous definitions and conclusions such as Theorem \ref{intercderi}, \ref{naturalthem}, \ref{Limitcalculus} have directly justified the validity of each step of the following proofs.
\begin{corr}[$\ln\mathrm{B}$]~ 
	\begin{enumerate}
		\item $\lim\limits_{n\to\infty}
		\left(
		\sum_{k=1}^{n}\frac{1}{n}-\ln n
		\right)=\gamma$;
		\item $\zeta(1+s)=\frac{1}{s}-\gamma+o(1)$;
		\item $(x!)'(0)=-\gamma$.
	\end{enumerate}
\end{corr}
\begin{proof}
	\begin{align*}
		\sum_{k=1}^{n}\frac{1}{n}-\ln n
		&=
		\ln(\mathrm{B}+n)+\gamma-\ln n\\
		&=
		\gamma+\ln(1+\frac{\mathrm{B}}{n})\to \gamma\\
		\zeta(1+s)
		&=
		\frac{\mathrm{B}^{-s}}{s}=\frac{1-s\ln\mathrm{B}+o(s)}{s}\\
		&=\frac{1}{s}-\gamma+o(1)\\
		\frac{(x!)'}{x!}
		&=
		\ln(\mathrm{B}+x)\\
		(x!)'(0)
		&=
		0!\ln\mathrm{B}=-\gamma.
	\end{align*}
\end{proof}
\begin{corr}[$\mathrm{B}\ln\mathrm{B}$]~ 
	\begin{enumerate}
		\item $n!\sim (\frac{n}{\mathrm{e}})^n\sqrt{2\pi n}$;
		\item $\zeta(s)=-\frac{1}{2}-s\ln\sqrt{2\pi}+o(s)$;
		\item $(-\frac{1}{2})!=-\frac{\sqrt{\pi}}{2}$.
	\end{enumerate}
\end{corr}
\begin{proof}
	\begin{align*}
		\ln n!
		&=
		(\mathrm{B}+n)\ln(\mathrm{B}+n)-n-\ln\sqrt{\frac{\mathrm{e}}{2\pi}}\\
		&=
		(\mathrm{B}+n)\ln n+(\mathrm{B}+n)\ln(1+\frac{\mathrm{B}}{n})-n-\ln\sqrt{\frac{\mathrm{e}}{2\pi}}\\
		&=
		(\frac{1}{2}+n)\ln n+\mathrm{B}^1+O(\frac{1}{n})-n-\ln\sqrt{\frac{\mathrm{e}}{2\pi}}\\
		&=
		(\frac{1}{2}+n)\ln n-n+(\frac{1}{2}-\ln\sqrt{\frac{\mathrm{e}}{2\pi}})+O(\frac{1}{n})\\
		\zeta(s)
		&=
		\frac{\mathrm{B}^{-s+1}}{s-1}=-\mathrm{B}^1+s(\mathrm{B}^1-\mathrm{B}\ln\mathrm{B})+o(s)\\
		&=
		-\frac{1}{2}-s\ln\sqrt{2\pi}+o(s)\\
		\ln(-\frac{1}{2})!
		&=
		\frac{2\mathrm{B}-1}{2}\ln\frac{2\mathrm{B}-1}{2}
		-\frac{1}{2}-\mathrm{B}\ln\mathrm{B}\\
		(\text{by Theorem \ref{naturalthem}, Lemma \ref{Multiplicationtheorem}})
		&=
		(2\cdot\frac{\mathrm{B}}{2}\ln\frac{\mathrm{B}}{2}-\mathrm{B}\ln\mathrm{B})-\frac{1}{2}-\mathrm{B}\ln\mathrm{B}\\
		&=
		-\mathrm{B}\ln2-\frac{1}{2}-\mathrm{B}\ln\mathrm{B}\\
		&=
		-\frac{1}{2}\ln2-\frac{1}{2}-\ln\sqrt{\frac{\mathrm{e}}{2\pi}}
		=\ln\sqrt{\pi}.
	\end{align*}
\end{proof}



\section{Some possible developments}

If we only apply the Bernoulli umbrae, then what we are actually doing is using that notation to simplify the application of the Abel-Plana formula. So the key points seem to be on the calculus between these umbrae.

For $A_i=(\mathcal{A}_i(z),\Omega_{a,b})$, we can define
\begin{equation*}
	A_1\times A_2\coloneqq
	(\mathrm{e}^{
		\frac{1}{z}
		\ln\frac{\mathcal{A}_1(z)}{\mathcal{A}_1(0)}
		\ln\frac{\mathcal{A}_2(z)}{\mathcal{A}_2(0)}
		},\Omega_{a,b})
\end{equation*}
where $A_1$ and $A_2$ are both invertible respect to the addition. It is again a umbrae, and gives a rare three-level arithmetic, i.e. ``$\times$'' distributes over ``$+$'', and ``$+$'' distributes over ``$[+]$''.

Let $\mu=(\mu_n)_{n\in\mathbb{N^+}}$ be a given sequence.
Suppose $\sum_{n=1}^{\infty}\mu_n z^n$ converges near the origin, such that $\mathcal{A}(z)=z^m\sum_{n=1}^{\infty}\mu_n\mathrm{e}^{nz}$
can be continuation to $\Omega_{a,b}$.
Denote
\begin{equation}
	(m\times\mathrm{D})+\mathrm{S}(\mu)\coloneqq
	(\mathcal{A}(z),\Omega_{a,b}).
\end{equation}
We can expect that for a suitable $f$, such as $f\in\mathcal{T}_{a,k;b,l}^{(p-1)}$, we have
\begin{equation}
	\sum_{n=1}^{\infty}\mu_n f(n)=f^{(-m)}(\mathrm{S}(\mu)).
\end{equation}
\noindent
This type of umbrae allows us to promote our previous treatment on Bernoulli umbrae further. For example,
\begin{enumerate}
	\item $(m\times\mathrm{D})+\mathrm{S}(\mu)=\mathrm{B}$,
	where $m=1$ and $\mu_n=-1$;
	\item $(m\times\mathrm{D})+\mathrm{S}(\mu)=\mathrm{E}$,
	where $m=0$ and $\mu_n=2\chi_1(n)$ where $\chi_1$ is the non-trivial Dirichlet character of modulo 4.
\end{enumerate}
Unfortunately, not like the Dirichlet character, the sequence $\mu$ seems cannot be replaced with a number theory sequence that is too irregular. For example, $\sum_{n=1}^{\infty}z^{n^2}$ is full of singularities on the entire unit circle.

Through the symmetry of the Fourier transformation, we are able to calculate $A^{-n}$ more directly, which is essentially the Ramanujan’s master theorem, i.e.
\begin{equation}
	\int_{0}^{\infty}
	\mathrm{e}^{-Ax}x^n\mathrm{d}x=
	n!A^{-n-1}.
\end{equation}
\noindent
It seems there is a chance to achieve more similar results.

We can use Bernoulli umbrae to deal with some series in \cite{Adamchik2005}. Although the process is relatively complicated, the approach is very straightforward.
\begin{lemm}For suitable $f$, we have
	\begin{align*}
		\sum_{n=1}^{z}\left(
		n\sum_{k=1}^{n}f(k)
		\right)
		&=
		-f^{(-3)}(\mathrm{B}+\mathrm{B}+z)+zf^{(-2)}(\mathrm{B}+\mathrm{B}+z)+f^{(-2)}(\mathrm{B}+\tilde{\mathrm{B}}+z)\\
		&\phantom{=}
		-\frac{z^2+z}{2}f^{(-1)}(\mathrm{B})-f^{(-2)}(\mathrm{B}+\tilde{\mathrm{B}})+f^{(-3)}(\mathrm{B}+\mathrm{B}),
	\end{align*}
	where $\tilde{\mathrm{B}}^n=\mathrm{B}_{n+1}$, and the two $\mathrm{B}$s in $\mathrm{B}+\mathrm{B}$ should be understood as different umbraes.
\end{lemm}
\begin{lemm}
	Suppose $\tilde{\mathrm{B}}^n=\mathrm{B}_{n+1},
	\tilde{\tilde{\mathrm{B}}}^n=\mathrm{B}_{n+2}$.
	\begin{enumerate}
		\item For suitable $f$, we have $f(\tilde{\mathrm{B}})=\mathrm{B}f(\mathrm{B}),
		f(\tilde{\tilde{\mathrm{B}}})=\mathrm{B}^2 f(\mathrm{B})$;
		\item $\mathrm{B}+\mathrm{B}=
		\mathrm{B}[+]
		(\mathrm{B}+\mathrm{D})[-]
		(\tilde{\mathrm{B}}+\mathrm{D})$;
		\item $\mathrm{B}+\tilde{\mathrm{B}}=
		(\mathrm{B}[+]
		(\tilde{\mathrm{B}}+\mathrm{D})[-]
		(\tilde{\tilde{\mathrm{B}}}+\mathrm{D}))+[\frac{1}{2}]$.
	\end{enumerate}
\end{lemm}

Finally, we can also see that if we allow the umbrae to be written as the difference between two analytic functions with disjoint domains, like the Fourier transform of the function $f$, we can eventually cover more general integral transforms.

\bibliography{umbral_calculus}

\begin{thebibliography}{10}

\bibitem{Adamchik2005}
V.~S. Adamchik.
\newblock The multiple gamma function and its application to computation of
  series.
\newblock {\em Ramanujan Journal. An International Journal Devoted to the Areas
  of Mathematics Influenced by Ramanujan}, 9(3):271--288, 2005.

\bibitem{DiBucchianico1995}
A.~Di~Bucchianico and D.~Loeb.
\newblock A selected survey of umbral calculus.
\newblock {\em Electronic Journal of Combinatorics}, 2:Dynamic Survey 3, 28,
  1995.

\bibitem{Ehrenpreis1958}
Leon Ehrenpreis.
\newblock Analytic functions and the {F}ourier transform of distributions.
  {II}.
\newblock {\em Transactions of the American Mathematical Society}, 89:450--483,
  1958.

\bibitem{Gessel2001}
Ira~M. Gessel.
\newblock Applications of the classical umbral calculus.
\newblock {\em Algebra Universalis 49 (2003), 397-434}, August 2001.

\bibitem{Gosper1993}
R.~William Gosper, Mourad E.~H. Ismail, and Ruiming Zhang.
\newblock On some strange summation formulas.
\newblock {\em Illinois Journal of Mathematics}, 37(2):240--277, 1993.

\bibitem{Grabiner1988}
Sandy Grabiner.
\newblock Convergent expansions and bounded operators in the umbral calculus.
\newblock {\em Advances in Mathematics}, 72(1):132--167, 1988.

\bibitem{Grabiner1989}
Sandy Grabiner.
\newblock Using {B}anach algebras to do analysis with the umbral calculus.
\newblock In {\em Conference on {A}utomatic {C}ontinuity and {B}anach
  {A}lgebras ({C}anberra, 1989)}, volume~21 of {\em Proc. Centre Math. Anal.
  Austral. Nat. Univ.}, pages 170--185. Austral. Nat. Univ., Canberra, 1989.

\bibitem{Ismail1997}
Mourad E.~H. Ismail and Dennis Stanton.
\newblock Classical orthogonal polynomials as moments.
\newblock {\em Canadian Journal of Mathematics. Journal Canadien de
  Math\'{e}matiques}, 49(3):520--542, 1997.

\bibitem{Ismail1998}
Mourad E.~H. Ismail and Dennis Stanton.
\newblock More orthogonal polynomials as moments.
\newblock In {\em Mathematical essays in honor of {G}ian-{C}arlo {R}ota
  ({C}ambridge, {MA}, 1996)}, volume 161 of {\em Progr. Math.}, pages 377--396.
  Birkh\"{a}user Boston, Boston, MA, 1998.

\bibitem{Mueller2005}
Markus M\"{u}ller and Dierk Schleicher.
\newblock How to add a non-integer number of terms, and how to produce unusual
  infinite summations.
\newblock {\em Journal of Computational and Applied Mathematics},
  178(1-2):347--360, 2005.

\bibitem{Mueller2010}
Markus M\"{u}ller and Dierk Schleicher.
\newblock Fractional sums and {E}uler-like identities.
\newblock {\em Ramanujan Journal. An International Journal Devoted to the Areas
  of Mathematics Influenced by Ramanujan}, 21(2):123--143, 2010.

\bibitem{Mueller2011}
Markus M\"{u}ller and Dierk Schleicher.
\newblock How to add a noninteger number of terms: from axioms to new
  identities.
\newblock {\em American Mathematical Monthly}, 118(2):136--152, 2011.

\bibitem{Rota1994}
G.-C. Rota and B.~D. Taylor.
\newblock The classical umbral calculus.
\newblock {\em SIAM Journal on Mathematical Analysis}, 25(2):694--711, 1994.

\bibitem{Zeilberger1980}
Doron Zeilberger.
\newblock Some comments on {R}ota's umbral calculus.
\newblock {\em Journal of Mathematical Analysis and Applications},
  74(2):456--463, 1980.

\end{thebibliography}
\end{document}